\documentclass[11pt]{article}
\usepackage{amssymb}
\usepackage{amsmath}
\usepackage{amsfonts}
\usepackage{tikz}
\usepackage{pgflibraryarrows}
\usepackage{pgflibraryplotmarks}
\usepackage{hyperref}
\usepackage{setspace}
      
\usepackage[letterpaper,margin=2.5cm]{geometry}

\newcommand{\C}{\mathbb{C}}
\newcommand{\R}{\mathbb{R}}

\newcommand{\Z}{\mathbb{Z}}

\newtheorem{theorem}{Theorem}[section]
\newtheorem{remark}[theorem]{Remark}
\newtheorem{proposition}[theorem]{Proposition}
\newtheorem{lemma}[theorem]{Lemma}
\newtheorem{definition}[theorem]{Definition}
\newtheorem{corollary}[theorem]{Corollary}

\newtheorem{conjecture}[theorem]{Conjecture}

\newenvironment{proof}{{\bf Proof:}}{\hfill$\square$\vskip.5cm}
\newenvironment{proofof}{}{\hfill$\square$\vskip.5cm}

\newcommand{\spec}{\operatorname{spec}}

\newcommand{\OIP}{\Omega^{\text{\rm IP}}}
\newcommand{\ORW}{\Omega^{\text{\rm RW}}}
\newcommand{\gip}{\gamma^{\text{\rm IP}}}
\newcommand{\grw}{\gamma^{\text{\rm RW}}}
\newcommand{\gsep}{\gamma^{\text{\rm SEP}}}

\title{Asymptotics of the Spectral Gap for the Interchange Process on Large Hypercubes}

\author{Shannon Starr$^{1}$ and Matt Conomos$^{2}$\\
Department of Mathematics, University of Rochester\\ 
RC Box 270138,
Rochester, NY 14627\\[1pt]
${}^{1}$ {\tt sstarr at math dot rochester dot edu}\\
${}^{2}$ 
{\tt mconomos at mail dot rochester dot edu}
}

\date{20 September 2011}

\begin{document}

\onehalfspacing

\maketitle 

\begin{abstract}
We consider the interchange process (IP)  on the $d$-dimensional,
discrete hypercube of side-length $n$.
Specifically, we compare the spectral gap of the IP to the spectral
gap of the random walk (RW) on the same graph.
We prove that the two spectral gaps are asymptotically equivalent,
in the limit $n \to \infty$.
This result gives further supporting evidence for a conjecture of 
Aldous, that the spectral gap of the IP equals the spectral gap of the RW
on all finite graphs.
Our proof is based on an argument invented by Handjani and Jungreis, who proved Aldous's
conjecture for all trees.

\vspace{8pt}
\vskip .2 cm
\noindent
{\small \bf MCS numbers:} 82C22; 60K35
\end{abstract}

\maketitle

\section{Introduction}

This paper is concerned with a certain conjecture of Aldous regarding the interchange process (IP) and the random walk (RW)
on finite graphs.
The IP is related to the RW, and can be thought of as a graphical representation for the RW.
It is the process that Liggett calls the ``stirring process'' in Section VIII.4 of his monograph, \cite{Liggett}.
The difference between the IP and the RW is this: 
while for the RW you have one particle moving about the vertices of the graph,
in the IP there are many more particles.
Specifically, there is one particle at each vertex; all particles are distinguishable;
and at random times particles at two endpoints of an edge
interchange their positions.
So the IP is really a random walk on permutations.
The stirring process also gives a graphical representation for the symmetric exclusion process (SEP).

For a fixed graph, $G$, let us refer to the Markov generator of the IP and RW by the symbols
$\OIP(G)$ and $\ORW(G)$, respectively.
The IP is a richer stochastic process than the RW, which is reflected in the eigenvalues:
$\spec(\ORW(G)) \subseteq \spec(\OIP(G))$.
One important quantity derived from the eigenvalues is the spectral gap, which is the distance between
the two largest eigenvalues of the Markov generator.
The spectral gap equals the reciprocal of the ``relaxation time'',
and it is related to the mixing time.
See, for example,
Chapter 8 of Aldous and Fill's book \cite{AldousFill} or Section 7.2 of Peres's notes \cite{Peres},
or Diaconis's review article about the cut-off phenomenon \cite{Diaconis}.

Let us refer to $\gip(G)$ and $\grw(G)$
as the spectral gap for the IP and RW, respectively. 
Since $\OIP(G)$ has more eigenvalues than $\ORW(G)$, 
and since the top eigenvalues of both are the same (they equal $0$ since we consider the continuous-time versions of the models)
we know that $\gip(G) \leq \grw(G)$.
Naively, we would expect that there is strict inequality.
But Aldous conjectured that $\gip(G)$ always equals $\grw(G)$.
Aldous stated this conjecture precisely in his book with Fill \cite{AldousFill} (it is Open Problem 29 in Chapter 14, Section 5)  and it is also listed on a page of open problems
on his website (\url{http://www.stat.berkeley.edu/~aldous/Research/OP/index.html}).

Let us say that $G$ is an ``Aldous graph'' if the condition $\gip(G) = \grw(G)$ is verified.
Many special families of graphs have been proved to be Aldous graphs, of which the most general family is trees.
Specifically, Handjani and Jungreis proved that all trees are Aldous graphs in \cite{HandjaniJungreis}.
For us, this is a key result in the literature.
We will use their arguments in a fundamental way.

We are interested in discrete hypercubes $R^d_n = \{(x_1,\dots,x_d) \in \Z^d\, :\, 1\leq x_1,\dots,x_d\leq n\}$.
Our result for these graphs will be something weaker than Aldous's conjecture, but related to it.
Let us start by recalling an important result which pre-dates Aldous's conjecture.
This is Lu and Yau's bound for the spectral gap of the SEP, which we will denote $\gsep(R^d_n)$  \cite{LuYau}.
Using a general method, which is now known as ``Yau's martingale method'',
they proved that $\gsep(R^d_n) = O(n^{-2})$.
This is also the known decay rate of $\grw(R^d_n)$.

Yau's martingale method is very important, and much more general than this one example suggests.
It applies to many models other than the SEP.
In particular it does not require any symmetry to apply, although the SEP and IP
have a lot of symmetries, themselves.
It also gives bounds on the Logarithmic-Sobolev-Inequality-constant, which is usually harder
to bound than the spectral gap.
The Logarithmic-Sobolev-Inequality-constant is important in mathematical physics and probability,
and is closely related
to the mixing time. (See Chapter 8 of \cite{AldousFill}.)

Presumably Yau's martingale method also applies to the IP to prove that $\gip(R^d_n) = O(n^{-2})$.
Therefore, one would conclude that 
$$
\liminf_{n\to \infty} \gip(R^d_n)/\grw(R^d_n)\, >\, 0\, .
$$
Aldous's conjecture  would imply that $\gip(R^d_n)/\grw(R^d_n) = 1$ for all $n$.
Therefore, while Lu and Yau's result is not as strong as Aldous's conjecture, it 
is related.

Our main result is the following:
\begin{theorem}
\label{thm:main}
For every dimension, $d\geq 1$,
$$
\lim_{n \to \infty} \gip(R^d_n)/\grw(R^d_n)\, =\, 1\, .
$$
\end{theorem}
This is somewhat akin to proving Aldous's conjecture asymptotically, in the limit of large hypercubes.
We view this as additional evidence in favor of Aldous's conjecture.

\begin{remark}
Shortly after our preprint appeared on the arXiv, Ben Morris independently discovered a similar argument to prove a similar result.
There are some differences in the two papers which the interested reader may see by consulting \cite{Morris}.
\end{remark}

\subsection{Outline for proof}

\subsubsection{The HJKN method}

We can state the following ideas of the proof.
Our result mainly relies on a theorem of Handjani and Jungreis, from \cite{HandjaniJungreis}. 
They invented an argument to show that an increasing sequence $G_2,G_3,\dots$ 
are all Aldous graphs, if one can show that the spectral gaps of the random walks are 
non-increasing: $\grw(G_2) \geq \grw(G_3) \geq \dots$. 
Here we suppose that the number of vertices in $G_n$ is $n$, and we exclude the single-vertex graph
because the spectral gaps do not exist.
This induction argument was also independently re-discovered by Koma and Nachtergaele in a different context \cite{KomaNachtergaele}.
Therefore, we refer to the argument as the HJKN method, to acknowledge all four researchers.
This method is quite useful.
As Handjani and Jungreis say, ``Usually, the gap for the random walk process is much easier to 
compute or bound than that for the interchange process."
In fact $\ORW(G)$ is one version of the discrete Laplacian for the graph $G$.
It is the one most closely related to the Neumann Laplacian.

A problem arises if, for a sequence of graphs, it is not the case that $\grw(G_N)$ is non-increasing with $N$.
For example, for some families of graphs, it may happen that $\grw(G_N)$ decreases on average,
but may have some steps where it increases.
Let us call $G_N$ a ``local minimum'' if $\grw(G_N) = \min \{ \grw(G_2),\dots,\grw(G_N)\}$.
Then our main contribution to the HJKN method is to show that it implies that each local minimum is an Aldous graph.
Also, as is intuitively clear, $\gip(G_k) \geq \gip(G_N) = \grw(G_N)$ for each $k\leq N$.

Let us assume that $\grw(G_N) \to 0$ as $N \to \infty$.
Then we can enumerate the local minima as $G_{N_1}, G_{N_2}, \dots$,
with $N_1<N_2<\dots$.
Defining 
$$
N(k)\, :=\, \min \{N_i\, :\, N_i \geq k\}\, ,
$$
we have lower bounds $\gip(G_k) \geq \grw(G_{N(k)})$.
Of course, also $\gip(G_k) \leq \grw(G_k)$.

If we know that $\grw(G_N) \sim C N^{-p}$, as $N  \to \infty$, for some $p>0$,
and $C<\infty$, then it is easy to conclude that $N(k) \sim k$, as $k \to \infty$.
Then the upper and lower bounds imply that $\gip(G_N) \sim C N^{-p}$, as well.
This allows us to conclude that $\gip(G_N) \sim \grw(G_N)$, as $N \to \infty$.
This is what we want to prove for hypercubes of side-length $n$, in the limit $n \to \infty$.

\subsubsection{Asymptotics of the random walk spectral gap on ``approximate'' hypercubes}

Fix the dimension, $d\in \{1,2,\dots\}$.
Using the HJKN method, the proof of our main result boils down to 
constructing a sequence of graphs $G_2,G_3,\dots$, such that
\begin{enumerate}
\item
$G_{n^d} = R^d_n$ for each integer $n\geq 2$, and
\item $\grw(G_N) \sim C N^{-p}$,
for some $C<\infty$ and $p>0$.
\end{enumerate}
It is easy to see that $\grw(R^d_n) \sim \pi^2 n^{-2}$, because one can completely diagonalize
$\ORW(R^d_n)$, using the Fourier series and the ``method of images''.
Therefore, obtaining such a sequence is certainly feasible, with $C = \pi^2$ and $p=2/d$.

For the actual proof, 
there is a technical problem that needs to be solved: controlling the gap for the random walk
on graphs $G_N$ for $N$ between hypercubes, i.e., $n^d < N < (n+1)^d$.
Since such graphs have ``extra vertices'' one cannot diagonalize them exactly.
It is still easy to produce variational upper bounds, using the variational principle,
to prove that $\limsup_{N \to \infty} N^{2/d} \grw(G_N) \leq \pi^2$.
But, as is usually the case, obtaining lower bounds on the spectral gaps requires more work.
Some conditions are required for the graphs. 
For example, at the very least such graphs must be assumed to be connected, otherwise the gap equals $0$.

The key technical lemma for this part of the proof comes from a discrete version of the Trace theorem.
Recall that the Trace theorem applies to an open domain $U \subset\subset \R^d$,
such that $\overline{U}$ is compact,
and such that $\partial U := \overline{U} \setminus U$ is piecewise $C^1$.
The theorem states that there is a bounded linear transformation from $H^1(U)$ to $L^2(\partial U)$, 
which is just the restriction map in the case of functions in $C^2(\overline{U})$.
(See Theorem 1 in Section 5.5 of Evans's textbook on PDE's, \cite{Evans}.)
It is easiest to understand this theorem when $d=1$, and considering a function $u : [0,1] \to \R$
which is in $C^2([0,1])$.
Then obviously,
$|u(1)| = \left|\int_0^1 \frac{d}{dx}[x u(x)]\, dx\right|
\leq 2 (\|u\|_{L^2} + \|u'\|_{L^2})$, and a similar argument works to bound $u(0)$.
Thus the restriction map is bounded.

For the graph $G_N$, with $R_n^d\subseteq G_N\subseteq R_{n+1}^d$, the discrete Laplacian is,
to leading order, equal to $n^2 \ORW(G_N)$.
The discrete $H^1$ norm is then $\|f\|^2 + n^2 \langle f, \ORW(G_N) f \rangle$.
The discrete version of the Trace theorem should imply that, for an appropriate
notion of $\partial G_N$, it is the case that 
$$
\sum\nolimits_{x \in \partial G_N} |f(x)|^2\, \leq\, a_d\, n^{-1} \|f\|^2 + b_d\, n\, \langle f, \ORW(G_N) f\rangle\, ,
$$
where $a_d$ and $b_d$ are finite constants (depending on $d$ but not on $n$).
Then one can effectively ``prune'' the extra vertices  to reduce the graph back to $R^d_n$.

The notion of $\partial G_N$ we use is $G_N \setminus R^d_n$.
In one dimension, the discrete Trace theorem is proved just as for the usual Trace theorem,
except using the finite difference operator in place of the derivative.
This gives the constants $a_1=b_1=2$.
When $d$ is greater than 1, we can reduce back to the 1-dimensional case by making conditions on
the graph. 
We partition $R^d_{n+1} \setminus R^d_n$ into $d$ ``faces'':
$S^d_{n,k} = \{(x_1,\dots,x_d)\, :\, x_k=n+1\ \text{and}\ x_{k+1},\dots,x_d \leq n\}$,
for $k=1,\dots,d$.
We require that for each point, $(x_1,\dots,x_d) \in S^d_{n,k}\cap G_N$, the entire line,
$$
K^d_{n,k}(x)\, :=\, \{(y_1,\dots,y_d)\, :\, y_i=x_i\ \text{for $i\neq k$, and}\ y_k\leq n+1\}\, ,
$$
lies in $G_N$.
Then we can use the 1-dimensional discrete Trace theorem on the subgraph $K^d_{n,k}(x)$.
This results in the inequality above with $a_d = 2d$ and $b_d=2$.
As the reader can easily check, the condition we require does imply that $G_N$ is connected, 
so that $\grw(G_N)>0$, at least.

Finally, we want to point-out that in this part of the proof we are guided by an
important paper of Ramirez, Rider and Virag on
random matrix theory \cite{RamirezRiderVirag}.
Among various important contributions of that paper, they showed how to use Sobolev inequalities
to prove strong types of convergence of finite-difference type operators to differential operators.
In their case, they proved a compactness result similar to the Rellich-Kondrachov theorem.
In our case, we needed a discrete Trace theorem.

\section{Set-up}

To begin with, we consider a slightly more general setting than the one described
in the introduction.
This generality will be useful later, in the proofs.

For each $N\geq 1$, define $X_N$ to be the set $\{1,2,\dots,N\}$.
Let $S_N$ be the set of permutations of $N$, which we denote as $\pi = (\pi_1,\dots,\pi_N)$,
where $\{\pi_1,\dots,\pi_N\} = X_N$.

Let $\ell^2(X_N)$ and $\ell^2(S_N)$ be real, finite-dimensional Hilbert spaces, defined as follows.
As vector spaces, let $\ell^2(X_N)$ be the set 
of all functions $f : X_N \to \C$
and let $\ell^2(S_N)$ be the set of all functions $f:S_N \to \C$.
Also, take the standard $\ell^2$ inner-products:
$$
\langle f,g \rangle_{\ell^2(X_N)}\, =\, \sum\nolimits_{i \in X_N} f(i) \overline{g(i)}\quad \text{and}\quad
\langle f,g \rangle_{\ell^2(S_N)}\, =\, \sum\nolimits_{\pi \in S_N} f(\pi) \overline{g(\pi)}\, .
$$
Let $\text{\rm U}(X_N)$ and $\text{\rm U}(S_N)$ refer to the groups of all unitary operators
on $\ell^2(X_N)$ and $\ell^2(S_N)$, respectively.
There are unitary group representations $U_N : S_N \to \text{\rm U}(X_N)$
and $V_N : S_N \to \text{\rm U}(S_N)$ defined as $U_N(\pi) f(i) = f((\pi^{-1})_i)$
and $V_N(\pi) f(\pi') = f(\pi^{-1} \pi')$.

In general, let us write $1_E$ for the indicator function of $E$. 
Also, let us write $1$ for the constant function.

For each $i \in X_N$, there is a surjection $\phi_{N,i} : S_N \to X_N$ given by $\phi_{N,i}(\pi) = \pi_i$.
One can define the linear transformation $T_{N,i} : \ell^2(X_N) \to \ell^2(S_N)$ given by
$T_{N,i} f = f \circ \phi_{N,i}$; i.e., $T_{N,i} f(\pi) = f(\pi_i)$. Then
$$
\big[V_N(\pi) [T_{N,i} f]\big](\pi')\, 
=\, [T_{N,i} f](\pi^{-1}\pi')\,
=\, f((\pi^{-1})_{\pi'_i})\,
=\, [U_N(\pi) f](\pi'_i)\,
=\, \big[T_{N,i} [U_N(\pi) f]\big](\pi')\, .
$$
Therefore, $V_N(\pi) T_{N,i} = T_{N,i} U_N(\pi)$.
In other words, $T_{N,i}$ intertwines the two representations of $S_N$.
Also note that 
$$
(T_{N,i})^* f(j)\, =\, \left\langle f, 1_{\phi_{N,i}^{-1}\{j\}} \right\rangle_{\ell^2(S_N)}\, .
$$
From this it is easy to see that $(T_{N,i})^* T_{N,i} = (N-1)! I_{X_N}$, where $I_{X_N}$
denotes the identity operator on $\ell^2(X_N)$.
Among other things, this implies that $T_{N,i}$ is injective.

\begin{lemma}
\label{lem:contraction}
For every choice of $z_{\pi} \in \C$, for $\pi \in S_N$,
$$
\spec\left(\sum\nolimits_{\pi \in S_N} z_\pi U_N(\pi)\right)\, \subseteq\, 
\spec\left(\sum\nolimits_{\pi \in S_N} z_\pi V_N(\pi)\right)\, .
$$
\end{lemma}
\begin{proof}
Let $A = \sum_{\pi \in S_N} z_{\pi} U_N(\pi)$ and let $B = \sum_{\pi in S_N} z_{\pi} V_N(\pi)$.
Then if $\lambda \in \spec(A)$, there must be a function $f \in \ell^2(X_N) \setminus \{0\}$ such that $A f = \lambda f$.
Since $T_{N,i}$ is injective, we know that $T_{N,i} f$ is in $\ell^2(S_N) \setminus \{0\}$.
But since $T_{N,i}$ is an intertwiner, we have $B T_{N,i} f = T_{N,i} A f = \lambda T_{N,i} f$. So $\lambda \in \spec(B)$.
\end{proof}

Given any $\pi \in S_N$, let us define $\Delta_{N,\pi} : \ell^2(X_N) \to \ell^2(X_N)$
and $\widehat{\Delta}_{N,\pi} : \ell^2(S_N) \to \ell^2(S_N)$ as
$$
\Delta_{N,\pi}\, =\, -I_{X_N} + \frac{1}{2} U_N(\pi) + \frac{1}{2} U_N(\pi^{-1})\quad
\text{and}\quad
\widehat{\Delta}_{N,\pi}\, =\, -I_{S_N} + \frac{1}{2} V_N(\pi) + \frac{1}{2} V_N(\pi^{-1})\, ,
$$
where $I_{X_N}$ and $I_{S_N}$ represent the identity operators on $\ell^2(X_N)$
and $\ell^2(S_N)$, respectively.
These operators can be written as 
$$
\Delta_{N,\pi}\, =\, - \frac{1}{2} \big(I_{X_N}- U_N(\pi)\big)^* \big(I_{X_N}- U_N(\pi)\big)\quad \text{and}\quad
\widehat{\Delta}_{N,\pi}\, =\, - \frac{1}{2} \big(I_{S_N} - V_N(\pi)\big)^* \big(I_{S_N} - V_N(\pi)\big) \, .
$$
Therefore, they are both negative semi-definite operators.
I.e., $\langle f,\Delta_{N,\pi} f\rangle_{\ell^2(X_N)} \leq 0$ for any $f \in \ell^2(X_N)$,
and $\langle f,\widehat{\Delta}_{N,\pi} f\rangle_{\ell^2(S_N)} \leq 0$ for any $f \in \ell^2(S_N)$.

Recall that, given a finite state space $X$, a Markov generator for a continuous-time Markov process
on $X$ is an operator, $\Omega : \ell^2(X) \to \ell^2(X)$, satisfying the following two conditions:
\begin{itemize}
\item $\Omega 1=0$; and
\item $\langle 1_{\{x\}}, \Omega 1_{\{y\}} \rangle \geq 0$ for any $x,y \in X$ with $x\neq y$.
\end{itemize}
It is also easy to see that $\Delta_{N,\pi}$ and $\widehat{\Delta}_{N,\pi}$
are Markov generators for continuous-time Markov processes on $X_N$ and $S_N$, respectively.
A Markov generator for a continuous-time Markov process which is self-adjoint is called a Markov generator
for a continuous-time, reversible Markov process.
Both $\Delta_{N,\pi}$ and $\widehat{\Delta}_{N,\pi}$ have this property.

The set of Markov generators for continuous-time, reversible Markov processes form a cone.
Therefore, if we have any nonnegative function, $r : S_N \to [0,\infty)$, we can define operators,
$$
\Delta_N(r)\, =\, \sum\nolimits_{\pi \in S_N} r(\pi) \Delta_{N,\pi}\quad \text{and}\quad
\widehat{\Delta}_N(r)\, =\, \sum\nolimits_{\pi \in S_N} r(\pi) \widehat{\Delta}_{N,\pi}\, ,
$$
which are both Markov generators for continuous-time, reversible Markov processes.
By Lemma \ref{lem:contraction}, $\spec(\Delta_N(r)) \subseteq \spec(\widehat{\Delta}_N(r))$.
The spectral gaps of these operators are defined as
\begin{equation}
\label{eq:gapRW}
\gamma(\Delta_N(r))\, =\, 
\min\{ \langle f,-\Delta_N(r) f\rangle_{\ell^2(X_N)}\, :\, f \in \ell^2(X_N)\, ,\ \|f\|_{\ell^2(X_N)} = 1\ \text{and}\ \langle f,1 \rangle_{\ell^2(X_N)} = 0\}
\end{equation}
and
\begin{equation}
\label{eq:gapIP}
\gamma(\widehat{\Delta}_N(r))\, =\, 
\min\{ \langle f,-\widehat{\Delta}_N(r) f\rangle_{\ell^2(S_N)}\, :\, f \in \ell^2(S_N)\, ,\ \|f\|_{\ell^2(S_N)} = 1\ \text{and}\ \langle f,1 \rangle_{\ell^2(S_N)} = 0\}\, .
\end{equation}
Since $T_{N,i} 1 = 1$ for all $i \in X_N$, the fact that $\spec(\Delta_N(r)) \subseteq \spec(\widehat{\Delta}_N(r))$
immediately implies that
$$
\gamma(\widehat{\Delta}_N(r))\, \leq\, \gamma(\Delta_N(r))\, .
$$

Let $X_{N,2} = \{\{i,j\}\, :\, 1\leq i<j\leq N\}$ be the set of pairs.
A function, $q : X_{N,2} \to [0,\infty)$, is called a ``rate function''.
Given a rate function, we let 
$$
\ORW_N(q)\, =\, \sum\nolimits_{\{i,j\} \in X_{N,2}} q(\{i,j\}) \Delta_{N,(i,j)}\quad \text{and}\quad
\OIP_N(q)\, =\, \sum\nolimits_{\{i,j\} \in X_{N,2}} q(\{i,j\}) \widehat{\Delta}_{N,(i,j)}\, ,
$$
where $(i,j) \in S_N$ is the standard transposition.
One may observe that for any function $r:S_N \to [0,\infty)$, we have
$$
\Delta_N(r)\, =\, \ORW_N(q)\quad \text{where}\quad
q(\{i,j\})\, =\, \frac{1}{2} \big[(T_{N,i})^*r(j) + (T_{N,j})^* r(i)\big]\, .
$$
We define $\grw_N(q)$ and $\gip_N(q)$ to be $\gamma(\ORW_N(q))$ and $\gamma(\OIP_N(q))$, respectively.
Of course, $\gip_N(q) \leq \grw_N(q)$ because $\ORW_N(q)$ and $\OIP_N(q)$ are special cases of the $\Delta_N(r)$ and $\widehat{\Delta}_N(r)$,
considered before.

In this context, Aldous's conjecture is the following.
\begin{conjecture}[Aldous's conjecture for the IP]
For every $N\geq 2$ and every $q:X_{N,2} \to [0,\infty)$,
$$
\gip_N(q)\, =\, \grw_N(q)\, .
$$
\end{conjecture}

\begin{remark}
If $N=1$ then there is only one eigenvalue of both the RW and IP, counting multiplicity.
So there is no spectral gap in that case.
For this reason, we never consider the case $N=1$.
Also, when $N=2$, both operators are $2\times 2$ matrices, which coincide.
So $\gip_2(q) = \grw_2(q) = 2 q(\{1,2\})$.
\end{remark}

\begin{definition}
For $N\geq 2$,
we say that $q : X_{N,2} \to [0,\infty)$ satisfies ``Aldous's condition'' if $\gip_N(q)=\grw_N(q)$.
\end{definition}

Consider a finite graph $G=(V,E)$, where $V$ is the vertex set and $E$ is the edge set.
We let $N = |V|$.
We consider edges to be unordered pairs $\{x,y\}$, and we do not allow ``loops''. 
So $E$ can be any subset of $\{\{x,y\}\, :\, x,y \in V\, ,\ x\neq y\}$.
Consider any enumeration of the vertices: $V = \{x_1,x_2,\dots,x_N\}$.
Then, there is a canonical choice of rate function, $q_G : X_{N,2} \to [0,\infty)$, associated to $G$:
$$
q_G(\{i,j\})\, =\, 1_{E}(\{x_i,x_j\})\, .
$$
We define $\ORW(G)$ and $\OIP(G)$ to be $\ORW_N(q)$ and $\OIP_N(q)$ for this $q$.
We also define $\grw(G)$ and $\grw(G)$ to be $\gamma(\ORW(G))$ and $\gamma(\OIP(G))$, respectively.

\begin{definition}
We say that $G$ is an ``Aldous graph'' if $q_G$ satisfies Aldous's condition.
\end{definition}

Let $d\in\{1,2,3,\dots\}$ be chosen as the dimension.
Consider $\Z^d$ to be the infinite graph with vertex set $\Z^d = \{(x_1,\dots,x_d)\, :\, x_1,\dots,x_d \in \Z\}$,
such that $\{x,y\}$ is an edge if and only if $\|x-y\|_1=1$, where $\|x\|_1 = |x_1| + \dots + |x_d|$.
This is the usual graph structure on $\Z^d$ making it the simple (hyper)cubic lattice.
Considering $\Z^d$ as the vertex set, given any finite subset $V \subset \Z^d$, let us consider $G$ to be the induced subgraph.
So $E = \{\{x,y\}\, :\, x,y \in V\, ,\ \|x-y\|_1=1\}$.
We will then simply write $V$, rather than $(V,E)$.
So we refer to $\grw(V)$ and $\gip(V)$ instead of $\grw(G)$ and $\gip(G)$.
In particular, we consider the discrete, $d$-dimensional hypercube of side-length $n$:
$$
R^d_n\, =\, \{(x_1,\dots,x_d) \in \Z^d\, :\, 1\leq x_1,\dots,x_d \leq n\}\, .
$$
With these preliminaries completed, we have all the necessary definitions to understand our main result:

\smallskip
\noindent
{\bf Theorem \ref{thm:main}}
{\it
For every dimension, $d\geq 1$,
$$
\lim_{n \to \infty} \gip(R^d_n)/\grw(R^d_n)\, =\, 1\, .
$$
}\\[-5pt]
As we mentioned in the introduction, this is somewhat similar to proving that for hypercubes, Aldous's conjecture holds `asymptotically'
in the limit that the side-length, $n$, approaches $\infty$.

\section{The HJKN Method}

In this section we will review Handjani and Jungreis's key theorem from their paper \cite{HandjaniJungreis},
and we will state a simple corollary of their methodology.
We call this the HJKN method.

\begin{definition}
Suppose that for each $k=2,\dots,N$, we have a rate function, $q_k : X_{k,2} \to [0,\infty)$.
We say that $q_2,\dots,q_N$ is increasing if $q_{k+1}(\{i,j\}) \geq q_k(\{i,j\})$ for every
$1\leq i<j\leq k<N$.
\end{definition}

\begin{theorem}[Handjani and Jungreis, 1995]
\label{thm:HJ}
Suppose that $q_k : X_{k,2} \to [0,\infty)$ is a rate function for each $k=2,\dots,N$ and that $q_2,\dots,q_N$ is increasing.
If, also,
$$
\grw_2(q_2) \geq \grw_3(q_3) \geq \dots \geq \grw_N(q_N)\, ,
$$
then $q_2,\dots,q_N$ all satisfy the Aldous condition.
\end{theorem}

\begin{remark}
Handjani and Jungreis did not state their theorem this way: they stated something slightly less general.
They restricted attention to the case $q_{k+1}(\{i,j\}) = q_k(\{i,j\})$ for every $1\leq i<j\leq k<N$.
But their proof works in this more general setting with no changes.
\end{remark}

We will not prove Handjani and Jungreis's theorem, here, since it can be found in \cite{HandjaniJungreis}.
It is their main result: Theorem 1.
The reader is urged to consult their paper, which we find to be highly readable,
and which contains other interesting results, as well.

As we mentioned in the introduction, essentially the same argument was re-discovered, although in the context of quantum
spin systems, by Koma and Nachtergaele in \cite{KomaNachtergaele}.
It is for this reason that we call the method the HJKN method, for Handjani, Jungreis, Koma and Nachtergaele.

The main goal for the rest of this section will be to prove the following simple corollary of Handjani and Jungreis's theorem.

\begin{corollary}
\label{cor:mHJKN}
Suppose that  $q_k : X_{k,2} \to [0,\infty)$ is a rate function for each $k=2,\dots,N$ and that $q_1,\dots,q_N$ is increasing.
If
$$
\grw_N(q_N)\, =\, \min_{2\leq k\leq N} \grw_k(q_k)\, ,
$$
then $q_N$ satisfies Aldous's condition and 
$$
\min_{2\leq k\leq N} \gip_k(q_k)\, =\, \gip_N(q_N)\, =\, \grw_N(q_N)\, .
$$
\end{corollary}

The key to proving the corollary will be to find rate functions $\tilde{q}_k : X_{k,2} \to [0,\infty)$ for $k=2,\dots,N$, such that:
first,
$\tilde{q}_2,\dots,\tilde{q}_N$ is increasing;
second,
$\grw_2(\tilde{q}_2) \geq \dots \geq \grw_N(\tilde{q}_N)$;
third,
$\tilde{q}_k(\{i,j\}) \leq q_k(\{i,j\})$, for all $1\leq i<j\leq k<N$; and,
finally,
$\tilde{q}_N(\{i,j\}) = q_N(\{i,j\})$, for all $1\leq i<j\leq N$.
We will need the following elementary lemma.

\begin{lemma}
\label{lem:positive}
For any $k\geq 2$, suppose that $q$ and $\tilde{q}$ are rate functions from $X_{k,2}$ to $[0,\infty)$ such that 
$\tilde{q} \leq q$, pointwise.
Then $\grw_k(q) \geq \grw_k(\tilde{q})$ and $\gip_k(q) \geq \gip_k(\tilde{q})$.
\end{lemma}
\begin{proofof}{\bf Proof of Lemma \ref{lem:positive}:}
Since $q(\{i,j\}) \geq \tilde{q}(\{i,j\})$ for all $1\leq i<j\leq k$, and since each $\Delta_{k,(i,j)}$ and $\widehat{\Delta}_{k,(i,j)}$ are negative
semi-definite, we see that
$$
\ORW_k(q) - \ORW_k(\tilde{q})\, =\, \sum\nolimits_{\{i,j\} \in X_{k,2}} [q(\{i,j\}) - \tilde{q}(\{i,j\})] \Delta_{k,(i,j)}
$$
and
$$
\OIP_k(q) - \OIP_k(\tilde{q})\, =\, \sum\nolimits_{\{i,j\} \in X_{k,2}} [q(\{i,j\}) - \tilde{q}(\{i,j\})] \widehat{\Delta}_{k,(i,j)}
$$
are both negative semi-definite.
In particular, this means that 
$$
\langle f,-\ORW_k(q) f\rangle_{\ell^2(X_k)}\, \geq\, \langle f,-\ORW_k(\tilde{q}) f\rangle_{\ell^2(X_k)}\, ,
$$
for all $f \in \ell^2(X_k)$, and
$$
\langle f,-\OIP_k(q) f\rangle_{\ell^2(X_k)}\, \geq\, \langle f,-\OIP_k(\tilde{q}) f\rangle_{\ell^2(S_k)}\, ,
$$
for all $f \in \ell^2(S_k)$.
Combined with the gap definitions in (\ref{eq:gapRW}) and (\ref{eq:gapIP}), this proves the lemma.
\end{proofof}

\begin{proofof}{\bf Proof of Corollary \ref{cor:mHJKN}:}
For each $k=2,\dots,N-1$, and all $t \in [0,1]$, let $\tilde{q}_{k,t}$ be the rate function on $X_{k,2}$ such that
$$
\tilde{q}_{k,t}(\{i,j\})\, =\, \begin{cases} t q_k(\{i,j\}) & \text{ if $k \in \{i,j\}$,}\\ q_k(\{i,j\}) & \text{ otherwise.} \end{cases}
$$
Since eigenvalues of matrices are continuous functions of the matrix entries, we see that $\grw_k(\tilde{q}_{k,t})$
is a continuous function of $t$.
Also, by Lemma \ref{lem:positive}, we see that $\grw_k(\tilde{q}_{k,t})$ is non-decreasing in $t$.
Note that $\grw_k(\tilde{q}_{k,0})=0$ because we can find an eigenvector with $0$-eigenvalue by taking
$$
f(i)\, =\, \begin{cases} \sqrt{(k-1)/k} & \text{ if $i = k$,} \\ -1/\sqrt{k(k-1)} & \text{ if $i=1,\dots,k-1$.} \end{cases}
$$
Referring to the gap definition of (\ref{eq:gapRW}),
it is trivial to check that $\langle f,1 \rangle_{\ell^2(X_k)}=0$.
But also, for $t=0$, we see that $\tilde{q}_{k,t}(\{i,k\})=0$ for all $1\leq i\leq k-1$.
So this function really does have eigenvalue equal to zero.
In other words, since the vertex $k$ is disconnected from $\{1,\dots,k-1\}$,
there are two stationary measures: the uniform measure on $\{1,\dots,k-1\}$ and the point-mass on $k$.

On the other hand, we know that, for $t=1$, we obtain $\grw_k(\tilde{q}_{k,t}) = \grw_k(\tilde{q}_{k,1}) \geq \grw_N(q_{N})$.
Since the function $t \mapsto \grw_k(\tilde{q}_{k,t})$ is continuous and non-decreasing with $t$, 
there is at least one $t_k$ such that $0\leq t_k\leq 1$ and $\grw_k(\tilde{q}_{k,t_k})=\grw_N(q_N)$.
For $2\leq k\leq N-1$, we let $\tilde{q}_k = \tilde{q}_{k,t_k}$ for this $k$.
We let $\tilde{q}_N = q_N$.
Note that, for $1\leq i,j\leq k-1$ we have $\tilde{q}_k(\{i,j\}) = q_k(\{i,j\})$, while for all $\{i,j\} \in X_{k,2}$, we have $\tilde{q}_k(\{i,j\}) \leq q_k(\{i,j\})$.
Transferring this property to $k+1$, assuming $2\leq k\leq N-1$, we see that, for $1\leq i<j\leq k$, we have
$$
\tilde{q}_{k+1}(\{i,j\})\, =\, q_{k+1}(\{i,j\})\, \geq\, q_k(\{i,j\})\, \geq\, \tilde{q}_k(\{i,j\})\, ,
$$
because we assumed $q_2,\dots,q_N$ was an increasing sequence.
Therefore, we conclude that $\tilde{q}_2,\dots,\tilde{q}_N$ is also an increasing sequence.
But $\grw_2(\tilde{q}_2)=\dots=\grw_N(\tilde{q}_N)$.
So, by Theorem \ref{thm:HJ}, this implies that $\tilde{q}_k$ satisfies Aldous's condition for each $k=2,\dots,N$.
In particular, $q_N=\tilde{q}_N$ satisfies Aldous's condition.

Also note that, by Lemma \ref{lem:positive}, we have $\gip_k(\tilde{q}_k) \leq \gip_k(q_k)$ because $\tilde{q}_k \leq q_k$, pointwise.
So
$$
\gip_k(q_k)\, \geq\, \gip_k(\tilde{q}_k)\, =\, \grw_k(\tilde{q}_k)\, =\, \grw_N(q_N)\, =\, \gip_N(q_N)\, .
$$
\end{proofof}

\begin{definition}
Given two real sequences, $(x_N)_{N=2}^{\infty}$ and $(y_N)_{N=2}^{\infty}$, such that $y_N > 0$ for all $N$,
we say that ``$x_N \sim y_N$, as $N \to \infty$,'' if 
$$
\lim_{N \to \infty} x_N/y_N\, =\, 1\, .
$$
\end{definition}

The following elementary fact is a corollary of the corollary.

\begin{proposition}
\label{prop:power}
Suppose that for each $k\geq 2$, there is a rate function $q_k : X_{k,2} \to [0,\infty)$ such that $q_2,q_3,\dots$ is increasing
and $\grw_2(q_2) \geq \grw_3(q_3) \geq \dots \geq \grw_k(q_k)  >0$ for each $k$. If there is a constant $C<\infty$ and an exponent $p>0$ such that 
$$
\grw_N(q_N)\, \sim\, C N^{-p}\, ,
$$ 
then also $\gip_N(q_N) \sim C N^{-p}$.
\end{proposition}
\begin{proof}
Choose $\epsilon \in (0,1)$.
Then there is some $N_0<\infty$ such that $(1-\epsilon) C N^{-p} \leq \grw_N(q_N) \leq (1+\epsilon) C N^{-p}$ for every $N\geq N_0$.
Note that $\min_{k\leq N_0} \grw_k(q_k)>0$ by assumption.
Since $N^{-p}$ converges to $0$, we can also find another $N_1$, such that $N_1 \geq N_0$ and such that 
$\grw_{N_1}(q_{N_1}) \leq \min_{k\leq N_0} \grw_k(q_k)$.
For each $N \geq N_1$, let us define 
$$
\grw_{\min}(N)\, :=\, \min_{k \in \{2,\dots,N\}} \grw_k(q_k)\, .
$$ 
Note that, since $\grw_{N_1}(q_{N_1}) \leq  \min_{k\leq N_0} \grw_k(q_k)$, we can actually take
$$
\grw_{\min}(N)\, :=\, \min_{k \in \{N_0,\dots,N\}} \grw_k(q_k)\, .
$$ 
Also, for each $N\geq N_1$, let us define
$$
K(N)\, =\, \max\{ k \in \{2,\dots,N\}\, :\, \grw_k(q_k) = \grw_{\min}(N)\}\, .
$$
Thus $K(N)$ is the largest $k\leq N$ such that $q_k$ gives a ``local min,'' and hence we are guaranteed that $q_k$ satisfies Aldous's condition
by Corollary \ref{cor:mHJKN}.
Note that for this local min 
$$
\grw_{K(N)}(q_{K(N)})\, =\, \grw_{\min}(N)\, \leq\, \grw_N(q_N)\, \leq\, (1+\epsilon) C N^{-p}\, .
$$
But also, $K(N) \geq N_0$ because we restricted the minimum to $k \in \{N_0,\dots,N\}$.
Therefore, 
$$
(1+\epsilon) C N^{-p}\, \geq\, \grw_{\min}(N)\, =\, \grw_{K(N)}(q_{K(N)})\, \geq\, (1-\epsilon) C \big(K(N)\big)^{-p}\, .
$$
This implies that $K(N) \geq \big((1-\epsilon)/(1+\epsilon)\big)^{1/p} N$.

Now, let us define $M(k) = \min\{M\, :\, M\geq N_1\, ,\ M\geq k\, \grw_{\min}(M) = \grw_M(q_M)\}$, for each $k\geq 2$.
This is nothing other than the smallest ``local min'' that occurs at or after $k$ (and greater than or equal to $N_1$ so that the asymptotic bounds apply).
Then, by Corollary \ref{cor:mHJKN}, $q_{M(k)}$ satisfies Aldous's condition, too, and
$$
\grw_{M(k)}(q_{M(k)})\, =\, \gip_{M(k)}(q_{M(k)})\, \leq\, \gip_k(q_k)\, \leq\, \grw_k(q_k)\, .
$$
But since $K(N) \geq \big((1-\epsilon)/(1+\epsilon)\big)^{1/p} N$, for $N\geq N_1$, we see that $M(k) \leq \lceil\big((1+\epsilon)/(1-\epsilon))^{1/p} k\rceil$,
for $k\geq N_1$.
Otherwise, we could take $N=\lceil\big((1+\epsilon)/(1-\epsilon))^{1/p} k\rceil$, and conclude that $K(N)$ is less than $M(k)$ and it is a ``local min'' no smaller than $k$
(contradicting  the definition of $M(k)$).
Therefore, we see that
$$
\gip_k(q_k)\, \geq\, \gip_{M(k)}(q_{M(k)})\, =\, \grw_{M(k)}(q_{M(k)})\, \geq\, C (1-\epsilon) [M(k)]^{-p}\, \succeq\, C\, \frac{(1-\epsilon)^2}{1+\epsilon}\, k^{-p}\, ,
$$
where $\succeq$ denotes asymptotically greater than: $a_k \succeq b_k$ if and only if $\liminf_{k \to \infty} a_k/b_k \geq 1$.
We already knew $\gip_k(q_k) \leq \grw _k(q_k) \leq (1+\epsilon) C k^{-p}$.
This shows that
$$
\liminf_{k \to \infty} \gip_k(q_k)/(C k^{-p})\, \geq\, \frac{(1-\epsilon)^2}{1+\epsilon}\quad \text{and}\quad
\limsup_{k \to \infty} \gip_k(q_k)/(C k^{-p})\, \leq\, 1+\epsilon\, .
$$
But $\epsilon \in (0,1)$ was arbitrary.
Taking $\epsilon \to 0^+$, this shows that 
$$
\lim_{k \to \infty} \gip_k(q_k)/(C k^{-p})\, =\, 1\, .
$$
\end{proof}


\section{A Discrete Trace Theorem}

We begin with an elementary lemma.

\begin{lemma}
\label{lem:trace1d}
For any $n\geq 1$, and for any $f \in \ell^2(\{1,\dots,n\})$, 
$$
|f(n+1)|^2\, \leq\, \frac{2}{n} \sum\nolimits_{k=1}^{n} |f(k)|^2 + 2n \sum\nolimits_{k=1}^{n} |f(k+1) - f(k)|^2\, .
$$
\end{lemma}
\begin{proof}
By a telescoping sum, it is easy to see that
$$
f(n+1)\, =\, \frac{1}{n} \sum\nolimits_{k=1}^{n} f(k) + \sum\nolimits_{k=1}^{n} \frac{k}{n} [f(k+1) - f(k)]\, .
$$
Therefore,
$$
|f(n+1)|\, \leq\, \frac{1}{n} \sum\nolimits_{k=1}^{n} |f(k)| + \sum\nolimits_{k=1}^{n} |f(k+1)-f(k)|\, .
$$
Applying the Cauchy-Schwarz inequality to each of the two sums, separately,
$$
|f(n+1)|\, \leq\, n^{-1/2}\left(\sum\nolimits_{k=1}^{n} |f(k)|^2\right)^{1/2} + n^{1/2} \left(\sum\nolimits_{k=1}^{n} |f(k+1)-f(k)|^2\right)^{1/2}\, .
$$
Since $(a+b)^2 \leq 2 a^2 + 2 b^2$, this then gives the result.
\end{proof}

The discrete Trace theorem is just a generalization of this basic inequality to higher dimensions, and more general graphs.
Suppose that the dimension, $d \in 1,2,\dots$, has been chosen and fixed.
Note that, given a finite subset $V \subset \Z^d_n$, we have
$$
\langle f, -\ORW(V) f\rangle_{\ell^2(V)}\, =\, \sum_{\substack{\{x,y\} \subseteq V\\ \|x-y\|_1=1}} \langle f, -\Delta_{(x,y)} f\rangle_{\ell^2(V)}\,
=\, \sum_{\substack{\{x,y\} \subseteq V\\ \|x-y\|_1 = 1}} |f(x) - f(y)|^2\, .
$$
For each $n\geq 1$,
define a ``simplicial decomposition" of $R^d_{n+1}$, as follows. For $k \in \{1,\dots,d\}$, define
$$
S^d_{n,k}\, =\, \{(x_1,\dots,x_d) \in \Z^d\, :\, 1\leq x_1,\dots,x_{k-1}\leq n+1\, ,\ x_k=n+1\, ,\ 1\leq x_{k+1},\dots,x_d\leq n\}\, .
$$
In Figure \ref{fig:one}, this is shown for two hypercubes: the $d=2$ hypercube, or square; and the $d=3$ hypercube, or cube.
Then $R^d_{n+1}$ can be written as a disjoint union
$$
R^d_{n+1}\, =\, R^d_n \cup \bigcup\nolimits_{k=1}^{d} S^d_{n,k}\, .
$$
Given $k \in \{1,\dots,d\}$ and $x \in S^d_{n,k}$,
define
$$
K^d_{n,k}(x)\, =\, \{(x_1,\dots,x_{k-1},j,x_{k+1},\dots,x_d)\, :\, 1\leq j\leq n+1\}\, .
$$
We make the following definition.

\begin{figure}
\centering
\begin{tikzpicture}[xscale=0.5,yscale=0.5]
\fill[gray!40!white] (4,0) rectangle (5,4);
\fill[gray!60!black] (0,4) rectangle (5,5);
\foreach \x in {0,1,...,5}
  \draw (\x,0) -- (\x,5); 
\foreach \y in {0,1,...,5}
  \draw (0,\y) -- (5,\y); 
\draw[very thick] (0,0) rectangle (4,4);
\draw[very thick] (0,4) rectangle (5,5);
\draw[very thick] (4,0) rectangle (5,4);
\draw (2.5,5) node[above] {$S^2_{4,2}$};
\draw (5,2) node[right] {$S^2_{4,1}$};
\draw(2,0) node[below] {$R^2_4$};
\end{tikzpicture}
\hspace{30pt}
\begin{tikzpicture}[xscale=0.4,yscale=0.4]
\draw[->,very thick] (1,1,1) -- (8,1,1) node[right] {\large $x_1$};
\draw[->,very thick] (1,1,1) -- (1,8,1) node[above] {\large $x_2$};
\draw[->,very thick] (1,1,1) -- (1,1,8.5) node[below left] {\large $x_3$};
\fill[gray!40!white] (6,1,1) -- (6,5,1) -- (6,5,5) -- (6,1,5) -- (6,1,1);
\fill[gray!90!black] (6,5,1) -- (6,6,1) -- (6,6,5) -- (6,5,5) -- (6,5,1);
\fill[gray!90!black] (1,6,1) -- (1,6,5) -- (6,6,5) -- (6,6,1) -- (1,6,1);
\fill[black!80!white] (1,1,6) rectangle (6,6,6);
\fill[black!80!white] (6,1,5) -- (6,6,5) -- (6,6,6) -- (6,1,6) -- (6,1,5);
\fill[black!80!white] (1,6,5) -- (6,6,5) -- (6,6,6) -- (1,6,6) -- (1,6,5);
\foreach \x in {1,...,5},
  \draw[thick] (\x,1,6) -- (\x,6,6);
\foreach \y in {1,...,5},
  \draw[thick] (1,\y,6) -- (6,\y,6);
\foreach \x in {1,...,6},
  \draw[thick] (\x,6,1) -- (\x,6,6);
\foreach \z in {1,...,6},
  \draw[thick] (1,6,\z) -- (6,6,\z);
\foreach \y in {1,...,5},
  \draw[thick] (6,\y,1) -- (6,\y,6);
\foreach \z in {1,...,6},
  \draw[thick] (6,1,\z) -- (6,6,\z);
\draw (6,3,1) node[right] {$S^3_{4,1}$};
\draw (3.5,6,1) node[above] {$S^3_{4,2}$};
\draw (3.5,1,6) node[below] {$S^3_{4,3}$};
\end{tikzpicture}
\caption{Two ``simplicial decompositions'': for $R^2_5$ and $R^3_5$ ($R^3_4$ is hidden).}
\label{fig:one}
\end{figure}
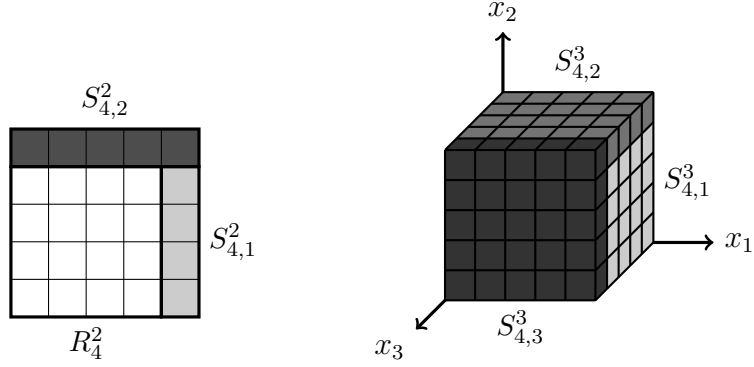

\begin{definition}
We say that $V \subseteq R^d_{n+1}$ is $R^d_n$-traceable if $K^d_{n,k}(x) \subseteq V$ for every $k \in \{1,\dots,d\}$ and $x \in V\cap S^d_{n,k}$.
\end{definition}

\begin{theorem}
\label{thm:trace}
Suppose that $V \subseteq R^d_{n+1}$ is $R^d_n$-traceable.
Then
\begin{equation}
\label{ineq:traced}
\sum_{x \in V \setminus R^d_n} |f(x)|^2\, \leq\, \frac{2d}{n} \|f\|^2_{\ell^2(V)} + 2n\langle f, -\ORW(V) f\rangle_{\ell^2(V)}\, .
\end{equation}
\end{theorem}
\begin{proof}
By Lemma \ref{lem:trace1d}, we know that for all $x \in V \cap S^d_{n,k}$
$$
|f(x)|^2\, \leq\, \frac{2}{n} \sum_{y \in K^d_{n,k}(x) \setminus \{x\}} |f(y)|^2
+ 2n \sum_{\substack{\{y,z\} \subset K^d_{n,k}(x)\\ \|y-z\|_1=1}} |f(y)-f(z)|^2\, .
$$
It is easy to see that every $y \in R^d_{n+1}$ is in at most $d$ subsets $K^d_{n,k}(x)$,
as $k$ varies over $\{1,\dots,d\}$ and $x$ varies over $S^d_{n,k}$.
Similarly, it is easy to see that every $\{y,z\}$ with $\|y-z\|_1=1$ is a subset of at most one
$K^d_{n,k}(x)$.
Therefore, summing over all $k \in \{1,\dots,d\}$, and all $x \in V \cap S^d_{n,k}$, we obtain
the result.
\end{proof}

\section{Completion of the Proof of the Main Theorem}

\begin{lemma}
\label{lem:lip}
Given two sets, $R^d_n \subseteq V \subseteq V' \subseteq R^d_{n+1}$, such that $V'$ is $R^d_n$-traceable,
\begin{equation}
\label{ineq:gapV'V}
\grw(V')\, \geq\, \frac{\left(1 - 2d n^{-1} - |V'\setminus V| \cdot |V|^{-1}\right)\grw(V)}{1 + 2 n \grw(V)}\, .
\end{equation}
\end{lemma}
We wish to remark on this inequality. We are thinking of the case that $V$ is actually $R^d_n$ so that we can easily calculate $\grw(V)$.
We actually want lower bounds on the spectral gap $\grw(V')$, as we have written it.
The reason is that we can easily obtain upper bounds for $\grw(V')$ due to the standard Rayleigh-Ritz variational formula, by trying a good variational
formula.
But lower bounds are harder to obtain.

\begin{proof}
The spectral gap is defined as in (\ref{eq:gapRW}), which we rewrite as 
$$
\grw(V')\, =\, 
\min\{ \langle f,-\ORW(V') f\rangle_{\ell^2(V')}\, :\, f \in \ell^2(V')\, ,\ \|f\|_{\ell^2(V')} = 1\ \text{and}\ \langle f,1 \rangle_{\ell^2(V')} = 0\}\, .
$$
But, actually, the minimum is attained at an eigenvector $f$ such that $-\ORW(V') f = \grw(V') f$.
Letting $f$ be this vector, we have
$$
\grw(V')\, =\, \frac{\langle f,-\ORW(V') f\rangle_{\ell^2(V')}}{\|f\|^2_{\ell^2(V')}}\, .
$$
Since $V \subseteq V'$, we can see that $-\ORW(V) \leq -\ORW(V')$, where $-\ORW(V)$ is extended to $\ell^2(V')$ naturally,
considering $V$ as a subgraph of $V'$.
This is the key to various inequalities: the extra terms contained in $[-\ORW(V')] - [-\ORW(V)]$ may all be expressed as positive semi-definite matrices.
Hence,
$$
\grw(V')\, =\,  \frac{\langle f,-\ORW(V') f\rangle_{\ell^2(V')}}{\|f\|^2_{\ell^2(V')}}\, \geq\,  \frac{\langle f,-\ORW(V) f\rangle_{\ell^2(V')}}{\|f\|^2_{\ell^2(V')}}\, .
$$
Let us denote $f\restriction V$ as $f_V$, in order to rewrite 
$$
\langle f,-\ORW(V) f\rangle_{\ell^2(V')}\, =\, \langle f_V, -\ORW(V) f_V\rangle_{\ell^2(V)}\, .
$$
Then, using the Rayleigh-Ritz mini-max definition of $\grw(V)$ as the spectral gap, the smallest eigenvalue of $-\ORW(V)$ among the subspace of vectors orthogonal to 1,
we obtain
$$
\langle f_V, -\ORW(V)\,  f_V\rangle_{\ell^2(V)}\,
\geq\, \left[\|f_V\|_{\ell^2(V)}^2 - \frac{|\langle f_V, 1\rangle_{\ell^2(V)}|^2}{\|1\|_{\ell^2(V)}^2}\right]\, \grw(V)\, .
$$
Denoting $f_{V'\setminus V}$ for $f  \restriction (V'\setminus V)$, we also have
$$
\|f_V\|_{\ell^2(V)}^2\, =\, \|f\|_{\ell^2(V')}^2 - \|f_{V'\setminus V}\|_{\ell^2(V'\setminus V)}^2\, .
$$
%
Finally, we can write $\langle f_V, 1\rangle_{\ell^2(V)}$ as $\langle f, 1_V\rangle_{\ell^2(V')}$, and we know $\|1\|_{\ell^2(V)}^2 = |V|$.
So, we have
\begin{equation}
\label{eq:additional}
\grw(V')\, \geq\,
\left(1 -  \frac{ \|f_{V'\setminus V}\|_{\ell^2(V'\setminus V)}^2}{\|f\|^2_{\ell^2(V')}} 
- |V|^{-1} \frac{|\langle f, 1_V\rangle_{\ell^2(V')}|^2}{\|f\|^2_{\ell^2(V')}}\right) \grw(V)\, .
\end{equation}
But, since $\langle f,1 \rangle_{\ell^2(V)}=0$, we have
$$
\langle f,1_{V} \rangle_{\ell^2(V')}\, 
=\, \langle f,1 \rangle_{\ell^2(V')} - \langle f,1_{V'\setminus V}\rangle_{\ell^2(V')}\,
=\, -\langle f,1_{V'\setminus V}\rangle_{\ell^2(V')}\, .
$$
So, by the Cauchy-Schwarz inequality,
$$
|\langle f,1_{V} \rangle_{\ell^2(V')}|^2\, \leq\, \|f\|_{\ell^2(V')}^2 \|1_{V'\setminus V}\|_{\ell^2(V')}^2\, =\, |V'\setminus V| \cdot \|f\|_{\ell^2(V')}^2\, .
$$
Putting this in (\ref{eq:additional}), we have
$$
\grw(V')\, \geq\,
\left(1 -  \frac{ \|f_{V'\setminus V}\|_{\ell^2(V'\setminus V)}^2}{\|f\|^2_{\ell^2(V')}} 
- \frac{|V'\setminus V|}{|V|}\right) \grw(V)\, .
$$
Finally, we notice that, by Theorem \ref{thm:trace},
$$
\|f_{V'\setminus V}\|_{\ell^2(V'\setminus V)}^2\, =\, \sum_{x \in V'\setminus V} |f(x)|^2\, \leq\, \sum_{x \in V' \setminus R^d_n} |f(x)|^2\,
\leq\, \frac{2d}{n} \|f\|_{\ell^2(V')}^2 + 2n \langle f,-\ORW(V)f\rangle_{\ell^2(V')}\, .
$$
Since $f$ is an eigenfunction and $-\ORW(V') f = \grw(V') f$, this gives
$$
\|f_{V'\setminus V}\|_{\ell^2(V'\setminus V)}^2\, 
\leq\, \big(2dn^{-1} + 2n \grw(V)\big) \|f\|^2_{\ell^2(V')}\, .
$$
Putting this together with the last inequality for $\grw(V')$, we obtain
\begin{equation}
\label{ineq:gapVV'}
\grw(V')\, \geq\, \big(1 - 2dn^{-1} - 2n \grw(V') - |V'\setminus V| \cdot |V|^{-1}\big) \grw(V)\, .
\end{equation}
Solving this inequality for $\grw(V')$ gives the desired result.
\end{proof}

We can now prove the main theorem

\begin{proofof}{\bf Proof of Theorem \ref{thm:main}:}
Firstly note that we can choose an infinite sequence of subsets $V_2 \subset V_3 \subset \dots \subset \Z^d$,
such that $|V_N|=N$, and $|V_{N+1}\setminus V_N|=1$, and such that for every $n\geq 1$ and every $N \in \{n^d,n^d+1,\dots,(n+1)^d\}$
we have that $V_N \subseteq R^d_{n+1}$ and $V_N$ is $R^d_n$-traceable.
Implicit in this is the fact that $V_{n^d} = R^d_n$ for each $n=2,3,\dots$.
To see this, suppose that $n^d\leq V_N<(n+1)^d$ and that $V_N$ is $R^d_n$-traceable.
Also suppose that $R^d_n \subseteq V_N$.
Then, if one appends any vertex in $S^d_{n,k}$, not already in $V_N$, with minimal $k$, then $V_{N+1}$ will also be $R^d_n$-traceable.

Let us suppose that $V_N = \{x_1,\dots,x_N\}$, where $V_{N+1} \setminus V_N = \{x_{N+1}\}$, for each $N\geq 1$.
Then it is easy to see that $q_2,q_3,\dots$ is an increasing sequence of rate functions, where
$$
q_N(\{j,k\})\, =\, \begin{cases} 1 & \text{if $\|x_j-x_k\|_1=1$,}\\ 0 & \text{otherwise.}\end{cases}
$$
So, if we prove that $\grw(V_N) \sim C N^{-p}$ for some $C<\infty$ and $p>0$, then we can apply Proposition
\ref{prop:power}.
We will prove this, next, with $C = \pi^2$ and $p=2/d$.

For each $n\geq 1$ it is trivial to diagonalize $\ORW(R^d_n)$, using one of the standard discrete Fourier bases.
All the eigenvectors are of the form
$$
f(x)\, =\, f^{(d)}(x;k)\, :=\, \prod_{i=1}^{d} f^{(1)}(x_i;k_i)\, ,
$$
where $x = (x_1,\dots,x_d)$ is a point in $R^d_n$ and $k=(k_1,\dots,k_d) \in \{0,1,\dots,n-1\}^d$.
We have
$$
f^{(1)}(x_i;k_i)\, =\, \begin{cases} 
n^{-1/2} & \text{ if $k_i=0$, and}\\
(2/n)^{1/2} \cos(\pi k_i (x_i - \frac{1}{2})/n) & \text{ if $k_i=1,\dots,n-1$.}
\end{cases}
$$
The eigenvalue is 
$$
\lambda^{(d)}(k)\, =\, \sum_{i=1}^{d} \lambda^{(1)}(k_i)\, ,\quad \text{where}\quad
\lambda^{(1)}(k_i)\, =\, -4 \sin^2\left(\frac{\pi k_i}{2n}\right)\, .
$$
There are various ways to see that this calculation is correct. The simplest is to notice that
$f(x)$ can be extended to all of $\Z^d$, and the choice of $k$'s (as well as the fact that we took cosine
and not sine) is just what is required to insure that the contributions to the discrete Laplacian across the boundary
of $R^d_n$ are all zero.
One can also simply check that each of the $f_k(x)$'s are eigenvectors (most simply using the comment from the last sentence)
and that they form an orthonormal basis for $\ell^2(R^d_n)$.
Because of all this, we see that 
$$
\grw(R^d_n)\, =\, 4 \sin^2\big(\pi/(2n)\big)\, ,
$$
which is equal to $\lambda^{(d)}(k)$ for any of the $k$'s with $k_i=1$ for some $i \in \{1,\dots,d\}$
and $k_j=0$ for all $j \in \{1,\dots,i-1,i+1,\dots,d\}$.
Therefore, obviously, $\grw(R^d_n) \sim \pi^2 n^{-2}$.
Or, in other words, $\grw(R^d_n) \sim \pi^2 |R^d_n|^{-2/d}$, since $|R^d_n|=n^d$.

Now, suppose that $n\geq 1$ and that $n^d \leq N\leq (n+1)^d$.
As before, we suppose that $V_N$ is $R^d_n$-traceable.
Also, as stated before, we assume $R^d_n \subseteq V_N$.
Then, by Theorem \ref{thm:trace}, we see that
$$
\grw(V_N)\, \geq\, \frac{(1-2dn^{-1} - |V_N\setminus R^d_n| n^{-d}) \grw(R^d_n)}{1+2n \grw(R^d_n)}\, ,
$$
where we took $V=R^d_n$ and $V'=V_N$.
But, since $\grw(R^d_n) \leq \pi^2/n^2$, we then obtain
$$
\grw(V_N)\, \geq\, \frac{(1-2dn^{-1} - (2^d-1)n^{-1})}{1+2\pi^2n^{-1}}\, \grw(R^d_n)\, ,
$$
where we used the crude bound $|V_N\setminus R^d_n|n^{-d} \leq (1+n^{-1})^d -1\leq (2^d-1) n^{-1}$.

But, if we again appeal to Theorem \ref{thm:trace}, this time using $V=V_N$ and $V'=R^d_{n+1}$, then we see that
$$
\grw(V_N) \cdot \big(1-2dn^{-1} -2n\grw(R^d_{n+1})-|R^d_{n+1}\setminus V_N| \cdot |V_N|^{-1}\big)\, \leq\, \grw(R^d_{n+1})\, .
$$
Indeed, this follows directly from equation (\ref{ineq:gapVV'}).
So, using bounds similar to those just used, we obtain
$$
\grw(V_N)\, \leq\, \big(1-[2d + 2 \pi^2 + (2^d-1)]n^{-1}\big)^{-1}\, \grw(R^d_{n+1})\, .
$$
Putting the upper and lower bounds together, and using the fact that $\grw(R^d_n) \sim \pi^2 n^{-2} = \pi^2 |R^d_n|^{-2/d}$, it obviously follows that
$$
\grw(V_N)\, \sim\, \pi^2 N^{-2/d}\, ,
$$
as desired.
So applying Proposition \ref{prop:power}, the main theorem follows.
\end{proofof}

\section*{Acknowledgements}

This research was supported in part by a U.S.\ National Science Foundation
grant, DMS-0706927.
S.S.\ is grateful to Bruno Nachtergaele and Wolfgang Spitzer for useful
comments and helpful conversations,
and to Thomas M.~Liggett for discussions on Aldous's conjecture and especially for 
bringing the paper of Handjani and Jungreis to his attention.
He is grateful to Brian Rider for explaining some aspects of reference \cite{RamirezRiderVirag}.
We thank Gady Kozma for pointing out an error in our original proof of Lemma 5.1, and we are grateful to
an anonymous referee for many helpful corrections based on a careful reading.

\end{document}